\documentclass[reqno,12pt]{amsart}

\usepackage{amsmath}
\usepackage{amsfonts}
\usepackage{amssymb}
\usepackage{eufrak}
\usepackage{amscd}
\usepackage{amsthm}
\usepackage{amstext}
\usepackage[all]{xy}

   \theoremstyle{plain}
   \newtheorem{thm}{Theorem}
   
   \newtheorem{lem}[thm]{Lemma}
   
   \theoremstyle{definition}
   
   \newtheorem{defn}[thm]{Definition}
   
   \theoremstyle{remark}

 \title[Vector bundles from generalized pairs of cocycles]%
{Vector bundles from generalized pairs of cocycles}
\author{V. Manuilov, Chao You}
\address{Department of Mechanics and Mathematics, Moscow State
University, Leninskie Gory, Moscow, 119991, Russia}
\address{Harbin Institute of Technology, 92 West DaZhi str., Harbin, 150001 P. R. China \\ \vspace{0.5cm}}

\email{manuilov@mech.math.msu.su}
\email{youchao@hit.edu.cn}
\thanks{The first named author acknowledges partial support from
RFFI, grant No. 13-01-00232. 
}

\begin{document}



 \begin{abstract}
It is interesting to know, how far we can generalize the notion of a group-valued cocycle keeping the property to determine a bundle. We find a generalization for pairs of cocycles and show how these generalized pairs of cocycles can still determine vector bundles.

 \end{abstract}

\maketitle

\subsection{Introduction}

It is a standard triviality that one can use a group-valued cocycle to glue up locally defined trivial bundles into a bundle. It is also known that one can use almost cocycles in this construction. It is interesting to know, how far we can generalize the notion of a cocycle keeping the property to determine a bundle. It was shown recently in \cite{Man} that $K$-theory elements can be represented not necessarily by pairs of projections but by pairs satisfying weaker properties. We follow this to find a generalization for pairs of cocycles and to show how generalized pairs of cocycles can still determine vector bundles. 

\subsection{Vector bundles from generalized pairs of cocycles}

Let $X$ be a compact Hausdorff space, let $\{U_\alpha\}_{\alpha\in\Lambda}$ be a finite open covering of $X$ and let $\mathrm U(N)$ denote the unitary group of an $N$-dimensional Hermitian space. Let 
$$
g_{\alpha\beta}^{\pm}:U_\alpha\cap U_\beta\to \mathrm U(N)
$$ 
be maps satisfying $g^{\pm}_{\alpha\beta}=(g^{\pm}_{\beta\alpha})^{-1}$. Fix $\varepsilon>0$.

\begin{defn}
The maps $g^{+}_{\alpha\beta}$ and $g^{-}_{\alpha\beta}$ are an {\it $\varepsilon$-generalized pair of cocycles} if  
\begin{equation}\label{1}
\|(g^+_{\alpha\beta}(x)g^+_{\beta\gamma}(x)-g^+_{\alpha\gamma}(x))(g^+_{\gamma\delta}(x)-g^-_{\gamma\delta}(x))\|<\varepsilon;
\quad
\|(g^-_{\alpha\beta}(x)g^-_{\beta\gamma}(x)-g^-_{\alpha\gamma}(x))(g^+_{\gamma\delta}(x)-g^-_{\gamma\delta}(x))\|<\varepsilon
\end{equation}
for all $\alpha,\beta,\gamma,\delta\in\Lambda$ and for every $x\in X$, for which the maps involved are defined.

\end{defn}

Note that (\ref{1}) are symmetric: if they hold for all $\alpha,\beta,\gamma,\delta\in\Lambda$ then, by passing to the adjoints, one obtains 
$$
\|(g^+_{\delta\alpha}(x)-g^-_{\delta\alpha}(x))(g^+_{\alpha\beta}(x)g^+_{\beta\gamma}(x)-g^+_{\alpha\gamma}(x))\|<\varepsilon;
\quad
\|(g^+_{\delta\alpha}(x)-g^-_{\delta\alpha}(x))(g^-_{\alpha\beta}(x)g^-_{\beta\gamma}(x)-g^-_{\alpha\gamma}(x))\|<\varepsilon
$$
for all $x\in X$, for which the maps involved are defined.

Let $\{\varphi_\alpha^2\}_{\alpha\in\Lambda}$ be a partition of the unity subordinated to the given covering, $0\leq\varphi_\alpha(x)\leq 1$, $\operatorname{supp}\varphi_\alpha\subset U_\alpha$, $\sum_\beta\varphi^2_\beta(x)=1$ for all $x\in X$ and for all $\alpha\in\Lambda$. Recall that if $\{g_{\alpha\beta}\}$ is a cocycle, i.e. if $g_{\alpha\beta}g_{\beta\gamma}=g_{\alpha\gamma}$ then the formula 
$$
P=[P_{\alpha\beta}]_{\alpha,\beta\in\Lambda}, \qquad P_{\alpha\beta}(x)=\varphi_\alpha(x)\varphi_\beta(x)g_{\alpha\beta}(x)
$$
is known to determine the vector bundle over $X$. More precisely, $P$ is well-defined (even at the points, where the cocycle is not) and is a projection in a trivial vector bundle over $X$ with the image being the vector bundle defined by the cocycle $\{g_{\alpha\beta}\}$. We can use $\varepsilon$-generalized pairs of cocycles to form the same matrix-valued functions, but without hope that these functions would be projection-valued. Set 
$$
A_+=[(A_+)_{\alpha\beta}]_{\alpha,\beta\in\Lambda},\qquad (A_+)_{\alpha\beta}(x)=\varphi_\alpha(x)\varphi_\beta(x)g^+_{\alpha\beta}(x)
$$ 
and
$$
A_-=[(A_-)_{\alpha\beta}]_{\alpha,\beta\in\Lambda},\qquad (A_-)_{\alpha\beta}(x)=\varphi_\alpha(x)\varphi_\beta(x)g^-_{\alpha\beta}(x).
$$ 

\begin{lem}\label{lemma2}
If $g^\pm_{\alpha\beta}$ are an $\varepsilon$-generalized pair of cocycles then $A_+$ and $A_-$ are selfadjoint and
$$
\|(A_+-A_+^2)(A_+-A_-)\|<m\varepsilon,\quad \|(A_--A_-^2)(A_+-A_-)\|<m\varepsilon,
$$
where $m=|\Lambda|$.

\end{lem}
\begin{proof}
Selfadjointness is obvious. The norm of an $(\alpha,\delta)$-entry of the matrix $(A_\pm-A_\pm^2)(A_+-A_-)$ can be estimated as follows: 
\begin{eqnarray*}
&&\!\!\!\!\!\|\varphi_\alpha(x)\varphi_\delta(x)\sum_\gamma\varphi_\gamma^2(x)\Bigl(g^\pm_{\alpha\gamma}(x)-\sum_\beta\varphi^2_\beta(x) g^\pm_{\alpha\beta}(x)g^\pm_{\beta\gamma}(x)\Bigr)(g^+_{\gamma\delta}(x)-g^-_{\gamma\delta}(x))\|\\
&=&\|\varphi_\alpha(x)\varphi_\delta(x)\sum_\gamma\varphi_\gamma^2(x)\sum_\beta\varphi^2_\beta(x)\Bigl(g^\pm_{\alpha\gamma}(x)-g^\pm_{\alpha\beta}(x)g^\pm_{\beta\gamma}(x)\Bigr)(g^+_{\gamma\delta}(x)-g^-_{\gamma\delta}(x))\|\\
&<&|\varphi_\alpha(x)\varphi_\delta(x)\sum_\gamma\varphi_\gamma^2(x)\sum_\beta\varphi^2_\beta(x)|\cdot\varepsilon\leq\varepsilon.
\end{eqnarray*}
Then the norm of the $m$-dimensional matrix with entries of the norm less than $\varepsilon$ is less than $m\varepsilon$. 

\end{proof}

\begin{lem}
$\|A_+\|\leq m$, $\|A_-\|\leq m$.

\end{lem}
\begin{proof}
Each entry of these two matrix-valued functions is a contraction.

\end{proof}

Set $g(t)=t-t^2$, $h(t)=tg(t)$.

\begin{lem}\label{L4}
$\|g(A_+)-g(A_-)\|<2(m+3)\sqrt{m\varepsilon}$. 

\end{lem}
\begin{proof}
We estimate the norm at a fixed point $x\in X$. As the estimate will not depend on $x$, so it will hold uniformly on $X$. At $x$, the matrix $A_+-A_-$ can be diagonalized (with eigenvalues $\lambda_1,\ldots,\lambda_N$), and we shall write all matrices with respect to the basis of eigenvectors $e_1,\ldots,e_N$ of $A_+-A_-$. Let $H_1$ be the linear span of the eigenvectors $e_i$ such that $|\lambda_i|<\sqrt{m\varepsilon}$, $H_2=H_1^\perp$. In what follows, we write matrices as two-by-two matrices with respect to the decomposition $H_1\oplus H_2$. Then $A_+-A_-=C=\left(\begin{smallmatrix}a_1&0\\0&a_2\end{smallmatrix}\right)$, where $\|a_1\|<\sqrt{m\varepsilon}$ and all eigenvalues of $a_2$ satisfy $|\lambda_i|\geq\sqrt{m\varepsilon}$. Then it follows from $\|g(A_\pm)(A_+-A_-)\|<m\varepsilon$ that
$g(A_\pm)=\left(\begin{smallmatrix}a^\pm_{11}&a^\pm_{12}\\a^\pm_{21}&a^\pm_{22}\end{smallmatrix}\right)$ with the norms of $a^\pm_{12}$, $a^\pm_{21}$ and $a^\pm_{22}$ smaller than $\sqrt{m\varepsilon}$. So, 
$$
\|g(A_+)-g(A_-)\|<\|a^+_{11}-a^-_{11}\|+6\sqrt{m\varepsilon}.
$$ 

As
\begin{eqnarray*}
g(A_+)-g(A_-)&=&A_-+C-A_-^2-A_-C-CA_--C^2-A_-+A_-^2\\
&=&C-C^2-A_-C-CA_-,
\end{eqnarray*} 
so $a^+_{11}-a^-_{11}$ is the upper left corner of the right hand side, therefore,
$$
\|a^+_{11}-a^-_{11}\|\leq \|a_1\|+\|a_1^2\|+2\|A_-\|\cdot\|a_1\|<2(m+1)\sqrt{m\varepsilon},
$$
and we finally conclude that $\|g(A_+)-g(A_-)\|<2(m+3)\sqrt{m\varepsilon}$.

\end{proof}

Let $f$ and $\varphi$ be the functions on $\mathbb R$ given by 
$$
f(t)=\left\lbrace\begin{array}{cl}0,& \mbox{if\ }t\leq 0;\\
t,&\mbox{if\ }0\leq t\leq 1;\\1,&\mbox{if\ }t\geq 1\end{array}\right.
\quad \mbox{and} \quad 
\varphi(t)=\left\lbrace\begin{array}{cl}0,& \mbox{if\ }t\leq 0;\\t,&\mbox{if\ }t\geq 0.\end{array}\right.
$$ 
Set $B_\pm=f(A_\pm)$.

\begin{lem}\label{L5}
Let $a,b$ be selfadjoint matrices, $\|a\|\leq m$, $\|b\|\leq m$, $m\geq 1,$ $\|a-b\|<\delta$, $\delta\leq 1$. Then $\|\varphi(a)-\varphi(b)\|<2\ln 16m\sqrt{\delta}$.

\end{lem}
\begin{proof}
It is shown in \cite{Kato} that $\||a|-|b|\|<\frac{2}{\pi}(2+\ln\frac{\|a\|+\|b\|}{\|a-b\|})\|a-b\|$. So, $\||a|-|b|\|<2\ln 16m\sqrt{\delta}$. Since $\varphi(a)=\frac{a+|a|}{2}$, we have
$$
\|\varphi(a)-\varphi(b)\|=\frac{1}{2}\|a+|a|-b-|b|\|<\frac{\delta}{2}+\ln 16m\sqrt{\delta}<2\ln 16m\sqrt{\delta}.
$$

\end{proof}

\begin{lem}\label{L6}
There is a computable constant $C(m)$ depending on $m$ such that 
$$
\|g(B_+)-g(B_-)\|<C(m)\sqrt[4]{\varepsilon}, \qquad \|h(B_+)-h(B_-)\|<C(m)\sqrt[4]{\varepsilon}.
$$

\end{lem}
\begin{proof}
As $g(f(t))=\varphi(g(t))$, so
$g(B_\pm)=g(f(A_\pm))=\varphi(g(A_\pm))$. By Lemma \ref{L4}, $\|g(A_+)-g(A_-)\|<2(m+3)\sqrt{m\varepsilon}$. Applying Lemma \ref{L5} with $\delta=2(m+3)\sqrt{m\varepsilon}$, we get 
$$
\|g(B_+)-g(B_-)\|<2\ln 16m \sqrt{2(m+3)}\sqrt[4]{m\varepsilon}.
$$

Similarly, as 
$$
h(f(t))=f(t)g(f(t))=f(t)\varphi(g(t))=t\varphi(g(t)),
$$ 
so 
$$
h(B_\pm)=h(f(A_\pm))=f(A_\pm)g(f(A_\pm))=f(A_\pm)\varphi(g(A_\pm))=A_\pm\varphi(g(A_\pm))=A_\pm g(B_\pm). 
$$
Consider $g(A_\pm)$ written with respect to the basis of eigenvectors, and let $P_\pm$ be the projection onto the subspace spanned by the eigenvectors with non-negative eigenvalues. Then $\varphi(g(A_\pm))=P_\pm g(A_\pm)P_\pm=g(A_\pm)P_\pm$. Therefore, 
\begin{eqnarray*}
\|h(B_+)-h(B_-)\|&=&\|A_+g(B_+)-A_+g(B_-)+A_+g(B_-)-A_-g(B_-)\|\\
&\leq& m\|g(B_+)-g(B_-)\|+ \|(A_+ -A_-)g(B_-)\|\\
&=&m\|g(B_+)-g(B_-)\|+ \|(A_+ -A_-)g(A_-)P_-\|\\
&\leq&m\|g(B_+)-g(B_-)\|+ \|(A_+ -A_-)g(A_-)\|\\
&<&2m\ln 16m \sqrt{2(m+3)}\sqrt[4]{m\varepsilon}+m\varepsilon.
\end{eqnarray*}

\end{proof}

\begin{lem}\label{L7}
$\|(B_\pm-B_\pm^2)(B_+-B_-)\|<2C(m)\sqrt[4]{\varepsilon}$.

\end{lem}
\begin{proof}
\begin{eqnarray*}
\|(g(B_+))(B_+-B_-)\|&\leq&\|g(B_+)B_+-g(B_-)B_-\|+\|g(B_-)B_--g(B_+)B_-\|\\
&\leq&\|h(B_+)-h(B_-)\|+\|g(B_+)-g(B_-)\|\cdot\|B_-\|<2C(m)\sqrt[4]{\varepsilon}.
\end{eqnarray*}
A similar estimate holds for $g(B_-)(B_+-B_-)$.

\end{proof}

Let $\kappa(t)=\sqrt{t-t^2}$, $B_\pm\in C(X, M_N)$ matrix-valued functions such that 
\begin{enumerate}
\item
$0\leq B_\pm\leq 1$;
\item
$\|(B_\pm-B_\pm^2)(B_+-B_-)\|<\delta$;
\item
$\|B_+-B_+^2-B_-+B_-^2\|<\delta$
\end{enumerate}
for some $0<\delta<1$.
Set $Q=\left(\begin{matrix}1-B_+&\kappa(B_+)\\\kappa(B_+)&B_-\end{matrix}\right)$, $Q\in C(X,M_{2N})$. 

\begin{lem}\label{L8}
For $Q$ defined above, one has $\|Q-Q^2\|<3\sqrt{\delta}$.

\end{lem}
\begin{proof}
Let us estimate the norms of matrix entries of $Q-Q^2$.
$$
(Q-Q^2)_{11}=1-B_+-(1-B_+)^2-\kappa(B_+)^2=B_+-B_+^2-g(B_+)=0.
$$
$$
(Q-Q^2)_{12}=(Q-Q^2)^*_{21}=\kappa(B_+)-\kappa(B_+)(1-B_++B_-)=\kappa(B_+)(B_+-B_-). 
$$
So, using the $C^*$-property, we have
\begin{eqnarray*}
\|(Q-Q^2)_{12}\|^2&=&\|\kappa(B_+)(B_+-B_-)\|^2=\|(B_+-B_-)\kappa(B_+)^2(B_+-B_-)\|\\
&\leq& \|B_+-B_-\|\cdot\|g(B_+)(B_+-B_-)\|<2\delta,
\end{eqnarray*}
hence $\|(Q-Q^2)_{12}\|<\sqrt{2\delta}$.
$$
(Q-Q^2)_{22}=B_--B_-^2-\kappa(B_+)^2=g(B_-)-g(B_+),
$$
hence $\|(Q-Q^2)_{22}\|<\delta$. Then $\|Q-Q^2\|<3\sqrt{\delta}$.

\end{proof}

Recall that $Q$ is an almost projection if $\|Q-Q^2\|<\frac{1}{4}$. In this case there exists a projection $P$ with $\|Q-P\|<\frac{1}{2}$, which determines a vector bundle over $X$. If $P'$ is another projection with $\|Q-P'\|<\frac{1}{2}$ then $\|P'-P\|<1$, hence $P$ and $P'$ determine isomorphic vector bundles.

Thus we have the following statement.
\begin{thm}\label{thm1}
For sufficiently small $\varepsilon>0$, the almost projection $Q$ defined from an $\varepsilon$-generalized pair of cocycles determines a vector bundle over $X$ up to isomorphism.

\end{thm}
\begin{proof}
Almost everything is already proved above. The only remaining claim to be proved is independence of the vector bundle (or, equivalently, that of an almost projection $Q$) from the choice of the partition of unity, which is obvious.

\end{proof}
Let us denote the vector bundle from Theorem \ref{thm1} by $\xi(g^+,g^-)$. 

\subsection{Case of a pair of $\varepsilon$-almost cocycles}
Let us consider now the case when an $\varepsilon$-generalized pair of cocycles is a pair of $\varepsilon$-almost cocycles, namely if the maps $g^\pm_{\alpha\beta}$ satisfy the condition
\begin{equation}\label{ac}
\|g^\pm_{\alpha\beta}g^\pm_{\beta\gamma}-g^\pm_{\alpha\gamma}\|<\varepsilon.
\end{equation}
It is obvious that if $g^+_{\alpha\beta}$ and $g^-_{\alpha\beta}$ are $\varepsilon$-almost cocycles then the pair $(g^+_{\alpha\beta},g^-_{\alpha\beta})$ is a $2\varepsilon$-generalized pair of cocycles.

If the two families of maps $\{g^\pm_{\alpha\beta}\}$ are  $\varepsilon$-almost cocycles, i.e. satisfy (\ref{ac}) then the matrix-valued functions $A_\pm$ are almost projections, namely, as in the proof of Lemma \ref{lemma2}, one can check that $\|A_\pm-A_\pm^2\|<m\varepsilon$. Thus, if $m\varepsilon<\frac{1}{4}$ then $A_+$ and $A_-$ determine two vector bundles over $X$ up to an isomorphism. Let us denote these vector bundles by $\eta(g^+)$ and $\eta(g^-)$, respectively. We also have $\|A_\pm-B_\pm\|<m\varepsilon$. As $\|\kappa(B_+)\|=\|B_+-B_+^2\|^{1/2}<2\sqrt{m\varepsilon}$, so 
$$
\left\|Q-\left(\begin{matrix}1-A_+&0\\0&A_-\end{matrix}\right)\right\|<4\sqrt{m\varepsilon},
$$
hence the following claim holds true. 
\begin{lem}
For a pair of $\varepsilon$-almost cocycles $g^+$, $g^-$ and for any sufficiently small $\varepsilon>0$, we have $[\xi(g^+,g^-)]=[1]-[\eta(g^+)]+[\eta(g^-)]$ in $K^0(X)$. 

\end{lem}

\subsection{Generalized pairs of representations}

Let $G$ be a discrete countable group and let $\tilde{X}\to X$ be a principal $G$-bundle
over a nice (e.g. compact metric) space $X$. 
Let $G$ act on $l^2(G)$ by (left) multiplication by viewing $G$ as a set of units of $l^2(G)$ and
let $l_X$ be the “line bundle” ˜$\tilde{X}\times_G l^2(G)\to X$
induced by the diagonal action of $G$. One can cover $X$ by a finite family of open
sets $(U_\alpha)_{\alpha\in\Lambda}$ in such a way that the bundle $l_X$ is obtained by gluing the trivial
bundles $U_\alpha\times l^2(G)$ via a constant cocycle $\gamma_{\alpha\beta}\in G \subset l^2(G)$.

The finitely generated projective right Banach $C(X)\otimes l^2(G)$-module of continuous sections of $l_X$ is represented by
an idempotent
$$
P=(\varphi_\alpha\varphi_\beta\cdot \gamma_{\alpha\beta})_{\alpha,\beta}\in M_m(C(X))\otimes l^2(G).	
$$

Let $\pi:G\to \mathrm U(N)$ be a map from $G$ to the unitary group of a finitedimensional Hermitian space. Set 
$$
P_\pi=(\varphi_\alpha\varphi_\beta\cdot \pi(\gamma_{\alpha\beta}))_{\alpha,\beta}\in M_{m\times N}(C(X)).
$$
If $\pi$ is a representation then $P_\pi$ is obviously a projection. If $\pi$ is an almost representation then $P_\pi$ is an almost projection, i.e. $P_\pi-P_\pi^2$ is small, which means that $P_\pi$ still determines a vector bundle over $X$ (but only up to isomorphism). Namely, there exists a (non-unique) projection $Q$ close to $P_\pi$, which determines a vector bundle. In \cite{Mishchenko}, in the case when $X$ is a finite simplicial complex, the transition functions are constructed for the vector bundle determined by such projection $Q$, thus giving an explicit form for $Q$, which otherwise can be constructed only implicitly, as a result of applying to $P_\pi$ a continuous function equal to 0 and to 1 in some neighborhoods of 0 and of 1 respectively.

For a finite subset $F\subset G$ and for $\varepsilon>0$, we call two maps $\pi_\pm:G\to \mathrm U(N)$ {\it a pair of $(F,\varepsilon)$-generalized representations} if $\pi_\pm(g^{-1})=\pi_\pm(g)^{-1}$ for any $g\in F$ and if
\begin{equation}\label{e-rep}
\|(\pi_\pm(gh)-\pi_\pm(g)\pi_\pm(h))(\pi_+(\gamma)-\pi_-(\gamma))\|<\varepsilon\nonumber
\end{equation} 
for any $g,h,\gamma\in F$.

\begin{lem}
If $F$ is sufficiently great and if $(\pi_+,\pi_-)$ is a pair of $(F,\varepsilon)$-generalized representations then $(\pi_+(\gamma_{\alpha\beta}),\pi_-(\gamma_{\alpha\beta}))$ is an $\varepsilon$-generalized pair of cocycles, where $\gamma_{\alpha\beta}$ is the constant $G$-valued cocycle as above.

\end{lem}
\begin{proof}
Obvious.
\end{proof}

\subsection{Concluding Remarks}
Thus we have a construction of vector bundles from pairs of generalized representations. Note that if we start with (pairs of) genuine representations then we get locally flat vector bundles, which are not too interesting from the $K$-theoretic point of view. If we start with pairs of almost representations then we get the so-called almost flat vector bundles. In many cases, any element of $K^0(X)$ can be represented by an almost flat vector bundle, and this plays an important role in the study of the Novikov conjecture. We cannot evaluate now if our construction gives almost flat vector bundles in the general case. Positive answer would show that pairs of generalized representations are in some sense equivalent to the pairs of almost representations, and negative answer would give some information about the Novikov conjecture when $K^0(BG)$ has non-almost flat elements.

\end{document}